\documentclass[journal]{IEEEtran}
\usepackage{cite}
\usepackage{amsmath,amssymb,amsfonts}
\usepackage{graphicx}
\usepackage{textcomp}
\usepackage{mathtools}
\usepackage{bm}
\usepackage[linesnumbered, ruled]{algorithm2e}

\usepackage{subfigure}
\usepackage{xcolor}
\usepackage{color}
\usepackage{stfloats}
\usepackage{endnotes}
\usepackage{url}

\usepackage{amsthm}
\newtheorem{theorem}{Theorem}

\allowdisplaybreaks[4]
\newcommand{\T}{\top}

\usepackage{multirow}
\usepackage{array}
\usepackage{booktabs}
\graphicspath{{figure/}}
\usepackage{dsfont}

\usepackage{xcolor}

\newcommand{\minew}[1]{#1}

\usepackage{soul}

\allowdisplaybreaks[4]

\begin{document}

\title{Variable Splitting Methods for Constrained State Estimation in Partially Observed Markov Processes}

\author{Rui Gao,
        Filip Tronarp,
        and Simo~S\"arkk\"a,~\IEEEmembership{Senior Member,~IEEE}
\thanks{
R.~Gao and S.~S\"arkk\"a are with the Department of Electrical Engineering and Automation, Aalto University, Espoo, 02150 Finland (e-mail: \{rui.gao, simo.sarkka\}@aalto.fi). F.~Tronarp is with the Department of Computer Science, Tübingen University, Tübingen, 72076 Germany (e-mail: filip.tronarp@uni-tuebingen.de)
}}

\markboth{Journal of \LaTeX\ Class Files,~Vol.~14, No.~8, August~2015}
{Shell \MakeLowercase{\textit{et al.}}: Bare Demo of IEEEtran.cls for IEEE Journals}

\maketitle

\begin{abstract}

In this paper, we propose a class of efficient, accurate, and general methods for solving state-estimation problems with equality and inequality constraints. The methods are based on recent developments in variable splitting and partially observed Markov processes. We first present the generalized framework based on variable splitting, then develop efficient methods to solve the state-estimation subproblems arising in the framework. The solutions to these subproblems can be made efficient by leveraging the Markovian structure of the model as is classically done in so-called Bayesian filtering and smoothing methods. The numerical experiments demonstrate that our methods outperform conventional optimization methods in computation cost as well as the estimation performance.
\end{abstract}


\IEEEpeerreviewmaketitle

\section{Introduction}
\label{sec:itroduction}
\IEEEPARstart{M}{any} real-world applications in signal processing, such as target tracking, indoor positioning, and robotics, can be formulated as state estimation tasks for restoring the hidden states given a set of incomplete observations~\cite{Bar-Shalom+Li+Kirubarajan:2001,simo2013Bayesian}. Mathematically, the estimation task is formulated as inference in a statistical model, where the state of the dynamic system evolves under constraints which arise naturally from physical properties or model assumptions. The aim of this paper is to present methods for models which differ from these classical formulations in the sense that they contain additional inequality or equality constraints \cite{Simon2002Kalman,Aravkin2017Generalized} on the solution itself. For example, the exploitation of the trajectory geometry in ship tracking has been proven to be effective in enhancing the accuracy of the trajectory estimate \cite{davidson2016survey,Bella2009KB}. Imposing such constraints leads to more accurate or physically reasonable estimates, but also makes solving the problem significantly more challenging. 

In recent decades, it has become common to formulate the methodology for state estimation in stochastic systems as special cases of Bayesian smoothers \cite{simo2013Bayesian}. However, it is well-known that some of these algorithms can be used to efficiently solve optimization problems -- even non-convex ones -- which arise from finding the maximum a posteriori (MAP) estimate or similar point-estimates of the system state. For instance, Kalman smoother (KS) is equivalent to the batch least-square solution \cite{Barfoot2017State}; iterated extended Kalman smoothers (IEKS) can be seen as Gauss--Newton methods for computing MAP estimates \cite{Bell1993ekf,Bell1994smoother}. One main advantage of these methods is that in large data regime, due to leveraging of the Markov properties of the systems, they are computationally efficient~\cite{Gao2019ieks}.

Variable splitting methods such as Peaceman--Rachford splitting (PRS) \cite{Peaceman1955}, split Bregman method (SBM)~\cite{Goldstein2009Split}, and alternating direction method of multipliers (ADMM)~\cite{Boyd2011admm,Gao2018admm} are efficient optimization methods, which can be applied to many kinds of constrained optimization problems. The main idea is to convert an inequality constrained problem into an equality constrained problem by variable splitting. Unfortunately, because of the nature of the problems, the direct solution via subgradient methods will remain computationally too expensive when the number of data points is large. However, it turns out that when applied to dynamic models, the computational demand can be lowered by using filtering and smoothing type of methods \cite{Gao2019mlsp,simo2020LMIEKS}.

In this paper, we focus on state estimation with nonlinear equality and/or inequality constraints while leveraging the Markovian structure of the model. The main contribution of the paper is to develop a novel class of constrained smoother-based variable splitting methods which can be instantiated by adopting different smoothers as well as variable splitting methods (i.e., ADMM, PRS, SBM), gaining the benefits of both. Their combination to solve constrained state-estimation problems leads to effective methods that have not previously been considered. For the special cases of Gaussian-driven systems, we present the constrained Kalman smoother and the constrained iterated extended Kalman smoother, which arise in the update steps of variable splitting. Our experiments demonstrate promising performance of the methods.

\subsection{Problem Formulation}
\label{sec:problem_formulation}
We consider general probabilistic state-space models, also called partially observed Markov processes (see \cite{simo2013Bayesian}):
\begin{equation}\label{eq:genmodel}
\begin{aligned}
\mathbf{x}_{t} \sim p(\mathbf{x}_{t} \mid \mathbf{x}_{t-1}), \quad
\mathbf{y}_{t} \sim p(\mathbf{y}_{t} \mid \mathbf{x}_t),
\end{aligned}
\end{equation}
where $\mathbf{x}_{t} \in \mathbb{R}^{N_x}$ denotes an $N_x$-dimensional state of the system, $\mathbf{y}_{t} \in \mathbb{R}^{N_y}$ is an $N_y$-dimensional noisy measurement at the time step $t=1,\ldots,T$, $p(\mathbf{x}_{t} \mid \mathbf{x}_{t-1})$ is the transition density of the Markovian state process, and $p(\mathbf{y}_{t} \mid \mathbf{x}_t)$ is the conditional probability density of the measurement. The prior distribution of the state is given as $p(\mathbf{x}_1)$. A particularly important special case is a model of the form
\begin{equation}\label{eq:gaussian_model}
\begin{aligned}
\mathbf{x}_{t} =\mathbf{a}_t(\mathbf{x}_{t-1})+ \mathbf{q}_t, \quad
\mathbf{y}_{t} = \mathbf{h}_t(\mathbf{x}_t)+ \mathbf{r}_t,
\end{aligned}
\end{equation}
where $\mathbf{h}_t: \mathbb{R}^{N_x} \to \mathbb{R}^{N_y}$ is a measurement function and $\mathbf{a}_t: \mathbb{R}^{N_x} \to \mathbb{R}^{N_x}$ is a state transition function. The initial state $\mathbf{x}_1$ is assumed to be Gaussian with mean $\mathbf{m}_1$ and covariance $\mathbf{P}_1$. The errors $\mathbf{q}_t$ and $\mathbf{r}_t$ are assumed to be mutually independent zero-mean Gaussian random variables with known positive definite covariance matrices $\mathbf{Q}_t$ and $\mathbf{R}_t$.

The goal is to estimate the state sequence $\mathbf{x}_{1:{T}} = \left\{\mathbf{x}_t\right\}_{t=1}^T$ from the noisy measurement sequence $\mathbf{y}_{1:{T}} = \{\mathbf{y}_t\}_{t=1}^T$, but we also have a set of equality and inequality constraints on the state. To make the problem tractable, we replace the state constraints with constraints of the state estimate, which we choose to be the MAP estimate. We aim at computing
\begin{equation} \label{eq:general_function}
\begin{split}
 &\min_{\mathbf{x}_{1:T}}
- \log p(\mathbf{x}_1)
-  \sum_{t=1}^T \log p(\mathbf{y}_{t} \mid \mathbf{x}_t)
- \sum_{t=2}^T  \log p(\mathbf{x}_{t} \mid \mathbf{x}_{t-1})  \\
&{\mathrm{s.t.}}\
\mathbf{e}_t({\mathbf{x}}_t)  = \mathbf{0}, \quad
 \mathbf{c}_t({\mathbf{x}}_t) \leq \mathbf{0}, \quad t = 1,\ldots,T,
\end{split}
\end{equation}
where $\mathbf{e}_t: \mathbb{R}^{N_x} \to \mathbb{R}^{N_e}$ and $\mathbf{c}_t: \mathbb{R}^{N_x} \to \mathbb{R}^{N_c}$ are constraint functions. For the model \eqref{eq:gaussian_model}, we obtain a nonlinear quadratic minimization problem.
In particular, when the functions $\mathbf{h}_t$, $\mathbf{a}_t$ $\mathbf{e}_t$, and $\mathbf{c}_t$ are affine, the problem in \eqref{eq:general_function} is a quadratic equality/inequality constrained optimization problem which can be solved in closed form.

\subsection{Overview of Variable Splitting}
\label{sec:various_methods}

Consider a minimization optimization problem with equality and inequality constraints:
\begin{equation}\label{eq:vs_function}
\begin{split}
\begin{aligned}
\min_{\mathbf{x}} \, \theta(\mathbf{x}),  \quad
\mathrm{s.t.}\, \mathbf{e}({\mathbf{x}}) = \mathbf{0}, \, \mathbf{c}({\mathbf{x}}) \leq \mathbf{0},
\end{aligned}
\end{split}
\end{equation}
where $\theta({\mathbf{x}})$ is the cost function, and $\mathbf{e}({\mathbf{x}})$, $\mathbf{c}({\mathbf{x}})$ are constraint functions. Our approach to solving problems of the form \eqref{eq:vs_function} proceeds by introducing auxiliary constrained variables to separate the components in the cost function, which is called \textit{variable splitting} \cite{Splitting2017book}. More specifically, we introduce an additional variable $\mathbf{v}$ and a barrier function ${I}(\mathbf{v})$ to replace the original inequality constraint, which leads to an equality constrained optimization problem
\begin{equation}\label{eq:general_constrained_function}
\begin{split}
\min_{\mathbf{x}}\, \theta(\mathbf{x}), \quad
\mathrm{s.t.}\  &\mathbf{e}(\mathbf{x}) = \mathbf{0}, \, \mathbf{c}(\mathbf{x}) + \mathbf{v} = \mathbf{0}, \\
 & I(\mathbf{v})= \begin{cases}
0, &   \mathbf{v} \geq 0, \\
\infty, &     \text{otherwise}.
\end{cases}
\end{split}
\end{equation}
We then define the so-called augmented Lagrangian function by introducing Lagrangian multipliers and penalty parameters. The process alternates among the updates of the split variables. For solving~\eqref{eq:general_constrained_function}, ADMM, PRS, and SBM variable splitting optimization methods are discussed.

ADMM was developed in part to blend the decomposability of dual ascent with the superior convergence properties of the method of multipliers \cite{Boyd2011admm}. Given $\mathbf{x}^{(0)}$, $\mathbf{v}^{(0)}$, $\bm{\eta}^{(0)}$, and $ \bm{\zeta}^{(0)}$, ADMM solves the constrained optimization problem~\eqref{eq:general_constrained_function} via the iterative steps:
\begin{subequations}
	\label{eq:general-admm}
	\begin{align}
	\label{eq:x-primal-admm}
	\mathbf{x}^{(k+1)} &= \mathop{\arg\min}_{\mathbf{x}} \theta(\mathbf{x})
	+ \frac{\rho_1}{2} \left\|\mathbf{c}(\mathbf{x}) +  \mathbf{v}^{(k)} + \bm{\eta}^{(k)}/\rho_1 \right\|^2  \notag\\
	& \quad \quad \quad + \frac{\rho_2}{2}  \left\|  \mathbf{e}(\mathbf{x}) + \bm{\zeta}^{(k)}/\rho_2 \right\|^2  , \\
		\label{eq:v-admm}
	\mathbf{v}^{(k+1)} &= {\max} \left( \mathbf{0}, \, - \mathbf{c}(\mathbf{x}^{(k+1)}) - \bm{\eta}^{(k)} /\rho_1 \right), \\
		\label{eq:eta-admm}
	\bm{\eta}^{(k+1)} & = \bm{\eta}^{(k)} +   \rho_1 \left(\mathbf{c}(\mathbf{x}^{(k+1)}) + \mathbf{v}^{(k+1)} \right), \\
		\label{eq:zeta-admm}
	\bm{\zeta}^{(k+1)} &= \bm{\zeta}^{(k)} +  \rho_2 \, \mathbf{e}(\mathbf{x}^{(k+1)}) ,
	\end{align}
\end{subequations}
where $\bm{\eta}$, $\bm{\zeta}$ are Lagrange multipliers associated with the constraints, and $\rho_1,\rho_2 > 0$ are penalty parameters.

The PRS method \cite{Peaceman1955,He2014peaceman} is similar to ADMM, but computes the primal variable once and updates the Lagrange multiplier twice, by updating the intermediate multipliers $\bm{\eta}^{(k+\frac{1}{2})}$ and $\bm{\zeta}^{(k+\frac{1}{2})}$). The iteration can be written as
\begin{subequations}
	\label{eq:general-PRS}
	\begin{align}
	\label{eq:x-primal-PRS}
	\mathbf{x}^{(k+1)} &= \mathop{\arg\min}_{\mathbf{x}} \theta(\mathbf{x})
	+ \frac{\rho_1}{2} \left\|\mathbf{c}(\mathbf{x}) +  \mathbf{v}^{(k)} + \bm{\eta}^{(k)}/\rho_1  \right\|^2  \notag\\
	& \quad \quad \quad + \frac{\rho_2}{2}  \left\|  \mathbf{e}(\mathbf{x}) + \bm{\zeta}^{(k)}/\rho_2 \right\|^2  , \\
	\bm{\eta}^{(k+\frac{1}{2})} &= \bm{\eta}^{(k)} + \alpha_1 \rho_1 \left( \mathbf{c}(\mathbf{x}^{(k+1)}) + \mathbf{v}^{(k)} \right), \\
     \bm{\zeta}^{(k+\frac{1}{2})} &= \bm{\zeta}^{(k)} +  \alpha_2 \rho_2\,\mathbf{e}(\mathbf{x}^{(k+1)}),  \\
    \mathbf{v}^{(k+1)} &= {\max} \left( \mathbf{0}, \,  -\mathbf{c}(\mathbf{x}^{(k+1)}) - \bm{\eta}^{(k)}  /\rho_1 \right), \\
	\bm{\eta}^{(k+1)} &= \bm{\eta}^{(k+\frac{1}{2})} + \alpha_1  \rho_1 \left( \mathbf{c}(\mathbf{x}^{(k+1)}) + \mathbf{v}^{(k+1)} \right), \\
	\bm{\zeta}^{(k+1)} &= \bm{\zeta}^{(k+\frac{1}{2})} +  \alpha_2 \rho_2\,\mathbf{e}(\mathbf{x}^{(k+1)}),
	\end{align}
\end{subequations}
with the parameters $\alpha_1,\alpha_2\in(0,1)$.

In SBM~\cite{Goldstein2009Split}, we have updates for $k = 1, 2\ldots$, as follows:
\begin{subequations}
\label{eq:general-SBM}
\begin{align}
\label{eq:x-primal-SBM}
	\mathbf{x}^{(k+1)} &= \mathop{\arg\min}_{\mathbf{x}} \theta(\mathbf{x}) + \frac{\rho_1}{2} \left\|\mathbf{c}(\mathbf{x}) + \mathbf{v}^{(k)} + \bm{\eta}^{(k)} \right\|^2  \notag\\
& \quad \quad \quad + \frac{\rho_2}{2}  \left\|  \mathbf{e}(\mathbf{x}) + \bm{\zeta}^{(k)} \right\|^2  , \\
 \mathbf{v}^{(k+1)} &= {\max} \left( \mathbf{0}, \,  - \mathbf{c}(\mathbf{x}^{(k+1)}) -\bm{\eta}^{(k)} \right),
\end{align}
\end{subequations}
for $M$ times, and update the extra variable by
\begin{subequations}\label{eq:sbm_eta}
\begin{align}
	\bm{\eta}^{(k+1)} & = \bm{\eta}^{(k)} +    \mathbf{c}(\mathbf{x}^{(k+1)}) + \mathbf{v}^{(k+1)} , \\
\bm{\zeta}^{(k+1)} &= \bm{\zeta}^{(k)} +  \mathbf{e}(\mathbf{x}^{(k+1)}).
\end{align}
\end{subequations}
In particular, SBM is equivalent to the scaled ADMM \cite{Boyd2011admm} when the inner iteration number $M = 1$.

All these methods discussed above, ADMM, PRS, and SBM, can solve optimization problems with equality and inequality constraints. Although the methods are slightly different, their updates of the primal variable $\mathbf{x}$ are similar. However, when the dynamic system is described in terms of a partially observed Markov process, the minimization problems typically become very high-dimensional.

\section{The Proposed Approach}
\label{sec:proposed}

In this section, we first employ the generalized constrained smoother-based variable splitting framework. \minew{For linear and nonlinear dynamic systems, we present} the constrained Kalman smoother and the constrained iterated extended Kalman smoother, respectively.

\subsection{The General Framework}
\label{sec:general_framework}
The methods we employ here rely on the variable splitting methods. Let $\theta(\mathbf{x}_{1:T})$ be a family of cost functions
\begin{equation}\label{eq:theta_function}
\begin{split}
 \theta(\mathbf{x}_{1:T}) &=  - \sum_{t=1}^T \log p(\mathbf{y}_{t} \mid \mathbf{x}_t)  \\
 &\qquad - \sum_{t=2}^T \log p(\mathbf{x}_{t} \mid \mathbf{x}_{t-1})  - \log p(\mathbf{x}_1).
\end{split}
\end{equation}
For \minew{Gaussian-driven nonlinear systems} in \eqref{eq:gaussian_model}, the function $ \theta(\mathbf{x}_{1:T})$ has the form \minew{(see also \cite{Bell1994smoother})}
\begin{equation}\label{eq:theta_function_gaussian}
\begin{split}
&\theta(\mathbf{x}_{1:T}) =  \frac{1}{2} \sum_{t=1}^{T}  \left\| \mathbf{y}_t - \mathbf{h}_t(\mathbf{x}_t)  \right\|_{\mathbf{R}_t^{-1}}^2\\
&\quad + \frac{1}{2}\sum_{t=2}^{{T}} \left\|\mathbf{x}_{t}-\mathbf{a}_t(\mathbf{x}_{t-1}) \right\|_{ \mathbf{Q}_t^{-1} }^2
+\frac{1}{2} \| \mathbf{x}_1  -   \mathbf{m}_1 \|_{ \mathbf{P}_1^{-1} }^2, 
\end{split}
\end{equation}
\minew{where $\|\mathbf{x}\|_{\mathbf{R}}^2 = \mathbf{x}^\T \mathbf{R} \mathbf{x}$.}
The equality/inequality constrained optimization problem corresponding to \eqref{eq:general_constrained_function} is now formed as
\begin{equation}\label{eq:constrained_function}
\begin{split}
\min_{\mathbf{x}_{1:T}}\,  &\theta(\mathbf{x}_{1:T})  \\
\mathrm{s.t.}\  &\mathbf{e}_t(\mathbf{x}_t) = \mathbf{0}, \quad \mathbf{c}_t({\mathbf{x}_t}) + \mathbf{v}_t = \mathbf{0}, \quad t = 1,\ldots,T, \\
& I(\mathbf{v}_t) = \begin{cases}
0, &  \mathbf{v}_t \geq 0, \\
\infty, &     \text{otherwise}.
\end{cases}
\end{split}
\end{equation}
As discussed in Section \ref{sec:various_methods}, the unified steps of solving~\eqref{eq:constrained_function} are to alternate minimization with respect to $\mathbf{x}_{1:T}$, $\mathbf{v}_{1:T}$, $\bm{\eta}_{1:T}$, and $\bm{\zeta}_{1:T}$, which depend on the specific method we choose to use. All the $\mathbf{x}_{1:T}$ subproblems in ADMM, PRS, and SBM are of nonlinear least-square type. The $\mathbf{v}_{t}$, $\bm{\eta}_{t}$, and $\bm{\zeta}_{t}$ subproblems in \eqref{eq:general-admm}, \eqref{eq:general-PRS}, and \eqref{eq:general-SBM} remain the same as in batch setting, except that we do the updates for each $t = 1,\ldots,T$ separately.

When $T$ is extremely large, the computation in the $\mathbf{x}_{1:T}$ subproblems have high computational costs. Therefore, in the following, we explicitly leverage the Markov structure of the problems which enables the use of computationally efficient Bayesian recursive smoothers, aiming at computing the MAP estimate of a constrained partially observed Markov process. 

\subsection{Constrained Kalman Smoother (CKS) for Affine Systems}
\label{sec:affine_constraint}
In this section, we present the method for solving the constrained state-estimation problem in affine systems. This solution which is based on KS will later be used in the nonlinear filters and smoothers. Let us now assume that the model and constraint functions are affine
\begin{equation}\label{eq:model}
\begin{aligned}
\mathbf{a}_t(\mathbf{x}_{t-1}) &= \mathbf{A}_t \, \mathbf{x}_{t-1} + \mathbf{b}_t, \quad
&\mathbf{h}_t(\mathbf{x}_t) &=\mathbf{H}_t \, \mathbf{x}_t + \mathbf{g}_t, \\
\mathbf{c}_t (\mathbf{x}_t) &= \mathbf{C}_t \,{\mathbf{x}}_t + \mathbf{d}_t,\quad
&\mathbf{e}_t (\mathbf{x}_t) &= \mathbf{E}_t \, {\mathbf{x}}_t  + \mathbf{f}_t,
\end{aligned}
\end{equation}
where $\mathbf{A}_t$ and $\mathbf{H}_t$ are the transition and measurement matrices, $\mathbf{E}_t$, $\mathbf{C}_t$ are constraint matrices, and $ \mathbf{b}_t$, $ \mathbf{g}_t$, $\mathbf{f}_t$, $\mathbf{d}_t$ are given vectors.

If we apply, for example, the ADMM method to the optimization problem in \eqref{eq:constrained_function}, then in the affine case the $\mathbf{x}_{1:T}$ subproblem becomes
\begin{equation}\label{eq:x_kf_linear}
\begin{split}
&\mathbf{x}_{1:T}^\star =  \mathop{\arg\min}_{\mathbf{x}_{1:T}}
 \frac{1}{2}  \sum_{t=1}^{T}  \left\| \mathbf{y}_t  - \mathbf{H}_t \, \mathbf{x}_t   - \mathbf{g}_t  \right\|_{\mathbf{R}_t^{-1}}^2  \\
& + \frac{1}{2}\sum_{t=2}^{{T}} \|\mathbf{x}_{t}   -   \mathbf{A}_t \, \mathbf{x}_{t-1}   -  \mathbf{b}_t \|_{ \mathbf{Q}_t^{-1} }^2
+ \frac{\rho_2}{2} \sum_{t=1}^{{T}}   \left\|  \mathbf{E}_t \mathbf{x}_t + \mathbf{f}_t + \frac{\bm{\zeta}_t}{\rho_2}  \right\|^2
\\
& + \frac{\rho_1}{2} \sum_{t=1}^{{T}}
\left\|\mathbf{C}_t \mathbf{x}_t + \mathbf{d}_t  + \mathbf{v}_t + \frac{\bm{\eta}_t}{\rho_1} \right\|^2  + \frac{1}{2}\| \mathbf{x}_1   -  \mathbf{m}_1 \|_{ \mathbf{P}_1^{-1} }^2.
\end{split}
\end{equation}
Similarly to \cite{Gao2019ieks}, this minimization problem corresponds to the MAP state estimate of an affine state-space model, and this estimate can be computed using KS. We now define two artificial measurement noises $\bm{\sigma}_t$ and $\bm{\delta}_t$ with covariances $\mathbf{\Sigma}_t$, $\mathbf{\Delta}_t$, and two pseudo-measurements $\mathbf{z}_t$, $\mathbf{w}_t$:
\begin{equation}\label{eq:setting}
\begin{aligned}
\mathbf{\Sigma}_t &= \mathbf{I}/\rho_1, &\mathbf{z}_t& = -\mathbf{v}_t  - {\bm{\eta}_t}/\rho_1,\\
\mathbf{\Delta}_t &= \mathbf{I}/\rho_2, &\mathbf{w}_t& = - \bm{\zeta}_t/\rho_2,
\end{aligned}
\end{equation}
where $\mathbf{I}$ is an identity matrix. The solution to \eqref{eq:x_kf_linear} can then be computed by running KS on the constrained state-space model
\begin{equation} \label{eq:linsubprob}
\begin{aligned}
\mathbf{x}_t &= \mathbf{A}_t\mathbf{x}_{t-1} + \mathbf{b}_t + \mathbf{q}_t,
&\mathbf{y}_t &= \mathbf{H}_t \mathbf{x}_t + \mathbf{g}_t + \mathbf{r}_t,  \\
\mathbf{z}_t &= \mathbf{C}_t \mathbf{x}_t + \mathbf{d}_t + \bm{\sigma}_t,
&\mathbf{w}_t &= \mathbf{E}_t \mathbf{x}_t + \mathbf{f}_t + \bm{\delta}_t.
\end{aligned}
\end{equation}

\subsection{Constrained Iterated Extended Kalman Smoother (CIEKS)}
\label{sec:nonlinear_ieks}
\minew{
When the state-space model or constraint functions are nonlinear, we can replace the Kalman smoother (KS) with the iterated extended Kalman smoother (IEKS) \cite{Bell1994smoother}. As a nonlinear extension of KS, IEKS works by iteratively linearizing the nonlinear functions around a previous estimate.} At each iteration $i$, we form the affine approximations for the nonlinear functions $\mathbf{a}_t$, $\mathbf{h}_t$, $\mathbf{c}_t$, and $\mathbf{e}_t$, given by
\begin{equation}
\label{eq:eks_nonliear_appro}
\begin{aligned}
\mathbf{a}_t(\mathbf{x}_{t-1}) &\approx \mathbf{a}_t(\mathbf{x}_{t-1}^{(i)}) + \mathbf{J}_{a_t}(\mathbf{x}_{t-1}^{(i)}) (\mathbf{x}_{t-1} - \mathbf{x}_{t-1}^{(i)}), \\
\mathbf{h}_t(\mathbf{x}_t) &\approx \mathbf{h}_t(\mathbf{x}_t^{(i)}) + \mathbf{J}_{h_t}(\mathbf{x}_t^{(i)}) (\mathbf{x}_t - \mathbf{x}_t^{(i)}), \\
\mathbf{c}_t(\mathbf{x}_t) &\approx \mathbf{c}_t(\mathbf{x}_{t}^{(i)}) + \mathbf{J}_{c_t}(\mathbf{x}_{t}^{(i)}) (\mathbf{x}_{t} - \mathbf{x}_{t}^{(i)}), \\
\mathbf{e}_t(\mathbf{x}_t) &\approx \mathbf{e}_t(\mathbf{x}_{t}^{(i)}) + \mathbf{J}_{e_t}(\mathbf{x}_{t}^{(i)}) (\mathbf{x}_{t} - \mathbf{x}_{t}^{(i)}),
\end{aligned}
\end{equation}
where $\mathbf{J}_\phi$ denotes the Jacobian of $\phi(\mathbf{x})$ and $\mathbf{x}_t^{(i)}$ is the current estimate. When we use the affine approximations on the $\mathbf{x}_{1:T}$ subproblem, for example~$\eqref{eq:x-primal-admm}$, the state-space model \eqref{eq:linsubprob} can be obtained if we explicitly write
\begin{equation} \label{eq:eks_setting}
\begin{aligned}
\mathbf{A}_{t} &=  \mathbf{J}_{a_t}(\mathbf{x}_{t-1}^{(i)}),
&\mathbf{b}_t& =  \mathbf{a}_t(\mathbf{x}_{t-1}^{(i)}) - \mathbf{J}_{a_t}(\mathbf{x}_{t-1}^{(i)}) \,\mathbf{x}_{t-1}^{(i)}, \\
\mathbf{H}_{t} &=\mathbf{J}_{h_t}(\mathbf{x}_t^{(i)}),
&\mathbf{g}_t& =  \mathbf{h}_t(\mathbf{x}_t^{(i)}) - \mathbf{J}_{h_t}(\mathbf{x}_t^{(i)})\, \mathbf{x}_t^{(i)}, \\
\mathbf{C}_{t} &=\mathbf{J}_{c_t}(\mathbf{x}_t^{(i)}),
&\mathbf{d}_t& = \mathbf{c}_t(\mathbf{x}_t^{(i)}) - \mathbf{J}_{c_t}(\mathbf{x}_t^{(i)}) \,\mathbf{x}_t^{(i)},\\
\mathbf{E}_{t} &=\mathbf{J}_{e_t}(\mathbf{x}_t^{(i)}),
&\mathbf{f}_t& = \mathbf{e}_t(\mathbf{x}_t^{(i)}) - \mathbf{J}_{e_t}(\mathbf{x}_t^{(i)}) \, \mathbf{x}_t^{(i)}.
\end{aligned}
\end{equation}
As discussed in Section~\ref{sec:affine_constraint}, we can obtain the solution by running KS. Hence, the $\mathbf{x}_{1:T}$ subproblem is solved by iterating these steps. CIEKS is equivalent to constrained Gauss--Newton (cf.\ \cite{Bell1994smoother}), but the smoother here uses the Markov structure of the problem, which leads to a lower computational cost.

\subsection{Convergence Results}
\label{sec:theoretical}
In this section, we present theoretical results for the proposed methods which are derived from the combination of CKS/CIKES and ADMM.

In the affine case, we have the following theorem.
\begin{theorem}[Convergence of CKS-ADMM]
	\label{theorem_admm}
	Let $\mathbf{Q}_t$ and $\mathbf{P}_1$ be positive definite matrices. The sequence $\{\mathbf{x}_{1:T}^{(k)}, \mathbf{v}_{1:T}^{(k)}, \bm{\eta}_{1:T}^{(k)}, \bm{\zeta}_{1:T}^{(k)}\}$ generated by CKS-ADMM globally converges to a stationary point $(\mathbf{x}_{1:T}^\star, \mathbf{v}_{1:T}^\star, \bm{\eta}_{1:T}^\star,\bm{\zeta}_{1:T}^\star)$.
\end{theorem}
\begin{proof}
	When the conditions $\mathbf{Q}_t \succeq \mathbf{0}$, $ \mathbf{P}_1 \succeq \mathbf{0}$ are satisfied, substituting into the batch form, the cost function $\mathbf{\theta}(\mathbf{x}_{1:T})$ will be convex. Hence, the convergence results can be
	obtained from the standard ADMM convergence proof \cite{Boyd2011admm}.
\end{proof}

On the other hand, when all the functions and constraints are nonlinear, $\mathbf{\theta}(\mathbf{x}_{1:T})$ may not be convex, for this case we have the following local convergence theorem.

\begin{theorem}[Convergence of CIEKS-ADMM]
	\label{theorem:CIEKS-ADMM}
	Let $\theta(\mathbf{x})$ be prox-regular \cite{Poliquin1996Prox} with the constant $M_{\theta}$ and the Jacobian $\mathbf{J}_c$, $\mathbf{J}_e$ have full-column rank. Then there exists $\rho_1,\, \rho_2 > 0$ such that
	 the sequence $\left\{\mathbf{x}_{1:T}^{(k)}, \mathbf{v}_{1:T}^{(k)},\bm{\eta}_{1:T}^{(k)}, \bm{\zeta}_{1:T}^{(k)} \right\}$ generated by CIEKS-ADMM converges to a local minimum $(\mathbf{x}_{1:T}^\star, \mathbf{v}_{1:T}^\star, \bm{\eta}_{1:T}^\star,\bm{\zeta}_{1:T}^\star)$.
\end{theorem}

\begin{proof}
The proof is based on our paper \minew{\cite[Theorem 2]{Gao2019ieks}}, mutatis mutandis.
\end{proof}

\section{Experiments}
\label{sec:simulation}
In this section, we evaluate the performance of the proposed constrained smoother-based variable splitting methods, including CIEKS-PRS, CIEKS-SBM, and CIEKS-ADMM, which are examples of those combinations. Consider a four-dimensional ship tracking model~\cite{Bella2009KB}, where the ship velocity $(x_{1,t}$, $x_{3,t})$ and its position $(x_{2,t}$, $x_{4,t})$ are given by
$$
\mathbf{x}_t =
\left(
1,\, t ,\,  -\cos(t) ,\,  1.3- \sin(t)
\right)^\T.
$$
The measurements are captured by two stationary positions, which are located at $(0,0)^\T$ and $(2 \pi, 0)^\T$.
The transition function and the covariance are
\begin{equation}
\begin{split}
\mathbf{a}_t(\mathbf{x}_{t-1}) &=
\left(
x_{1,t-1}, \, x_{2,t-1} + x_{1,t-1} \Delta t, \right. \\
&\qquad \quad  \left. x_{3,t-1}, \,  x_{4,t-1} + x_{3,t-1} \Delta t
\right)^\T,\\
\mathbf{Q}_t &=
\begin{bmatrix}
\Delta t  &  {\Delta t^2 }/{2} & 0 & 0  \\
{\Delta t^2 }/{2}  & {\Delta t^3 }/{3}  & 0 & 0 \\
0  & 0 &  {\Delta t }  & {\Delta t^2 }/{2}  \\
0 &  0  & {\Delta t^2 }/{2}   & {\Delta t^3 }/{3} \nonumber
\end{bmatrix},
\end{split}
\end{equation}
with $T = 100$, $\Delta t = 2\pi/ T$.
The measurement function and the covariance are
$$
\mathbf{h}_t(\mathbf{x}_{t}) =
\begin{pmatrix}
\sqrt{x_{2,t}^2 + x_{4,t}^2 } \\
\sqrt{(x_{2,t} - 2\pi)^2 + x_{4,t}^2 }
\end{pmatrix}, \quad
\mathbf{R}_t=
\begin{pmatrix}
\tau^2 & 0 \\
0&\tau^2
\end{pmatrix},
$$
with $\tau = 0.25$.
We impose the inequality constraint
$
\mathbf{c}_t(\mathbf{x}_{t})  = 1.25 - \sin(x_{2,t} ) - x_{4,t} \le 0
$
into the model. We compare CIEKS-ADMM with the unconstrained estimate and the constrained robust Kalman-Bucy smoother (CKBS) \cite{Bella2009KB}. We set $\rho_1 = 1$ and the maximum number of iterations is $100$. The estimation results are plotted in Fig.~\ref{fig:constraint}. As can be seen, the results with the constraints are much closer to the ground truth than the unconstrained estimate. Due to visual similarity, we only plot the estimate of CIEKS-ADMM here.

Fig.~\ref{fig:compared_result} demonstrates the computational benefits of our CIEKS-based variable splitting methods, compared to the batch PRS, SBM, ADMM methods, and CKBS. The left-hand plot depicts the value of the function $\theta(\mathbf{x}_{1:T})$ versus the iteration number. We observe that all the methods converge in just a few iterations. Our methods, CIEKS-PRS, CIEKS-SBM, and CIEKS-ADMM, have the same convergence rate with the corresponding batch PRS, SBM, and ADMM. They have superior convergence properties over CKBS \cite{Bella2009KB}. Meanwhile, the constrained smoother-based variable splitting methods solve the problem fastest. The right-hand plot in Fig.~\ref{fig:compared_result} compares the running time (sec) of all the methods, confirming that our methods are able to keep this growth very mild, in contrast to the batch methods. The proposed methods take around 15 iterations to run, and are complete in about 0.2 seconds. The benefit of our methods is highlighted by the fact that they can efficiently solve constrained state-estimation problems with extremely large numbers of data points. Table~\ref{table1} shows that with increasing $T$ from $10^3$ to $10^6$, CKBS, PRS, SBM, and ADMM become significantly slower, and the CIEKS-based PRS, SBM, and ADMM yield significant speed improvements. Particularly, the batch methods run out of memory (`--') when $T \geq 10^5$.

\begin{figure}[tbh]
	\centerline{\includegraphics[width=7.5cm,height=5cm]
		{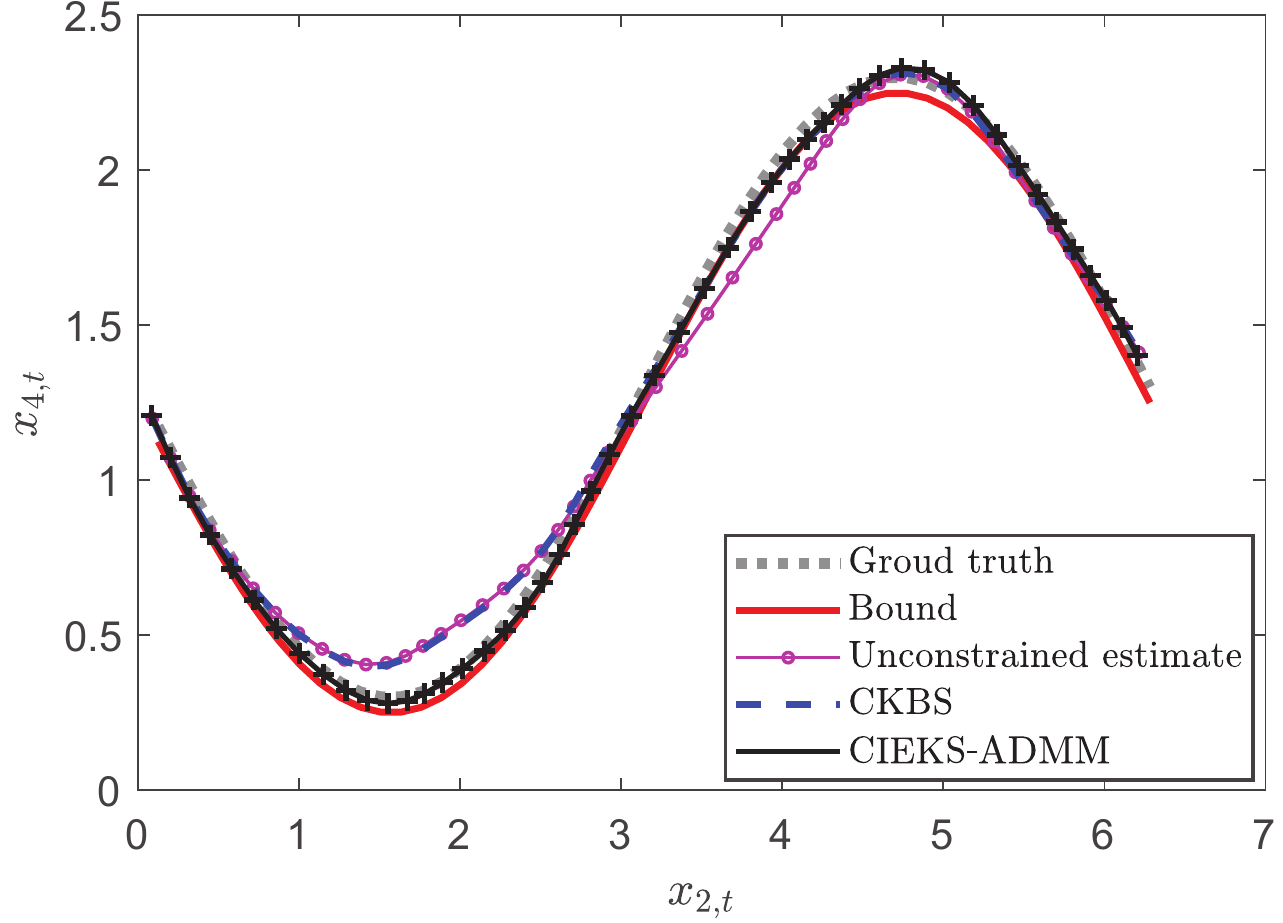}}
	\caption{Ground truth (dashed gray), bound (red line), and the estimates.}
	\label{fig:constraint}
\end{figure}
\begin{figure}[tbh]
		\begin{minipage}{0.49\linewidth}
	\centerline{\includegraphics[width=0.99\columnwidth]{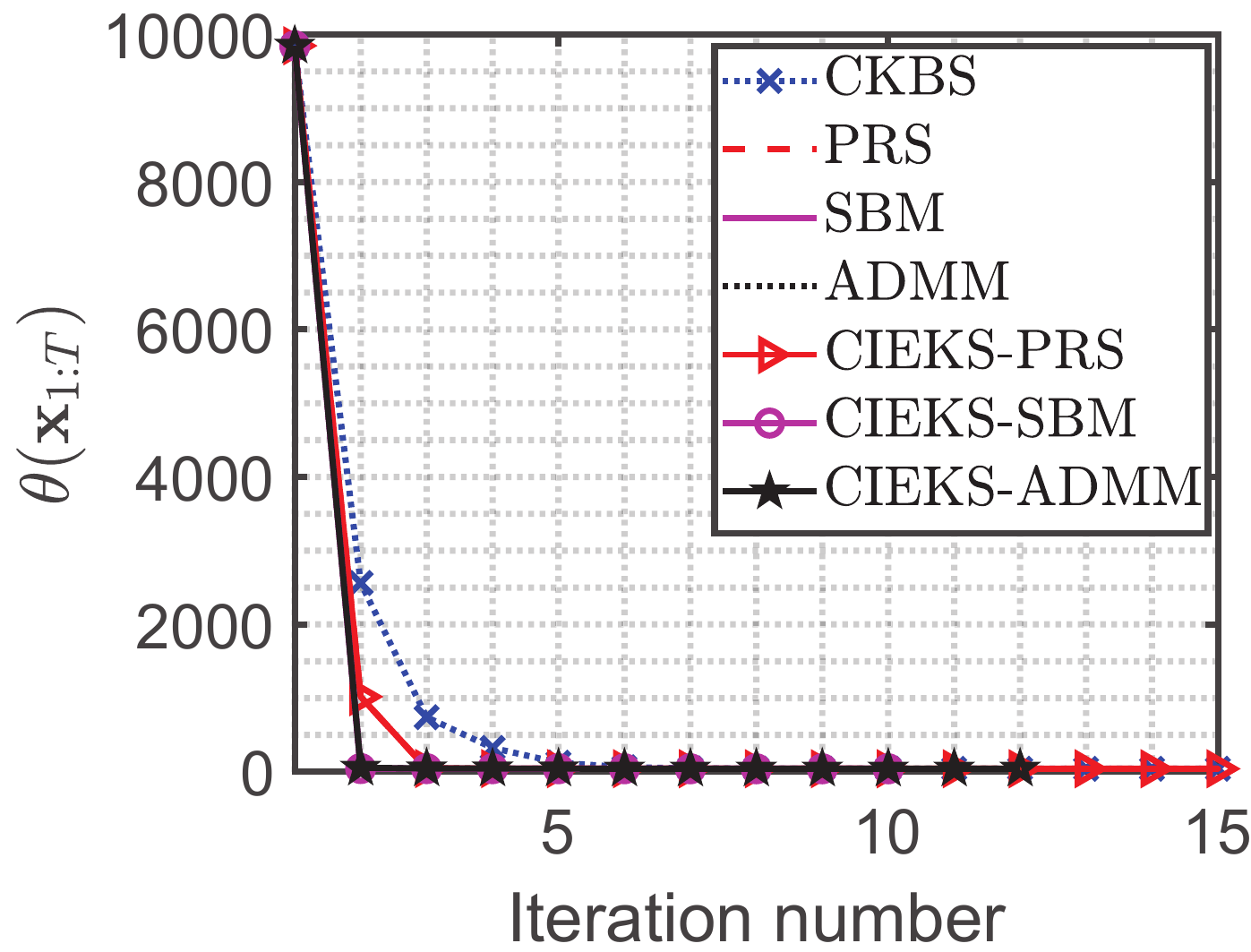}}
	\end{minipage}
	\begin{minipage}{0.49\linewidth}
	\centerline{\includegraphics[width=0.99\columnwidth]{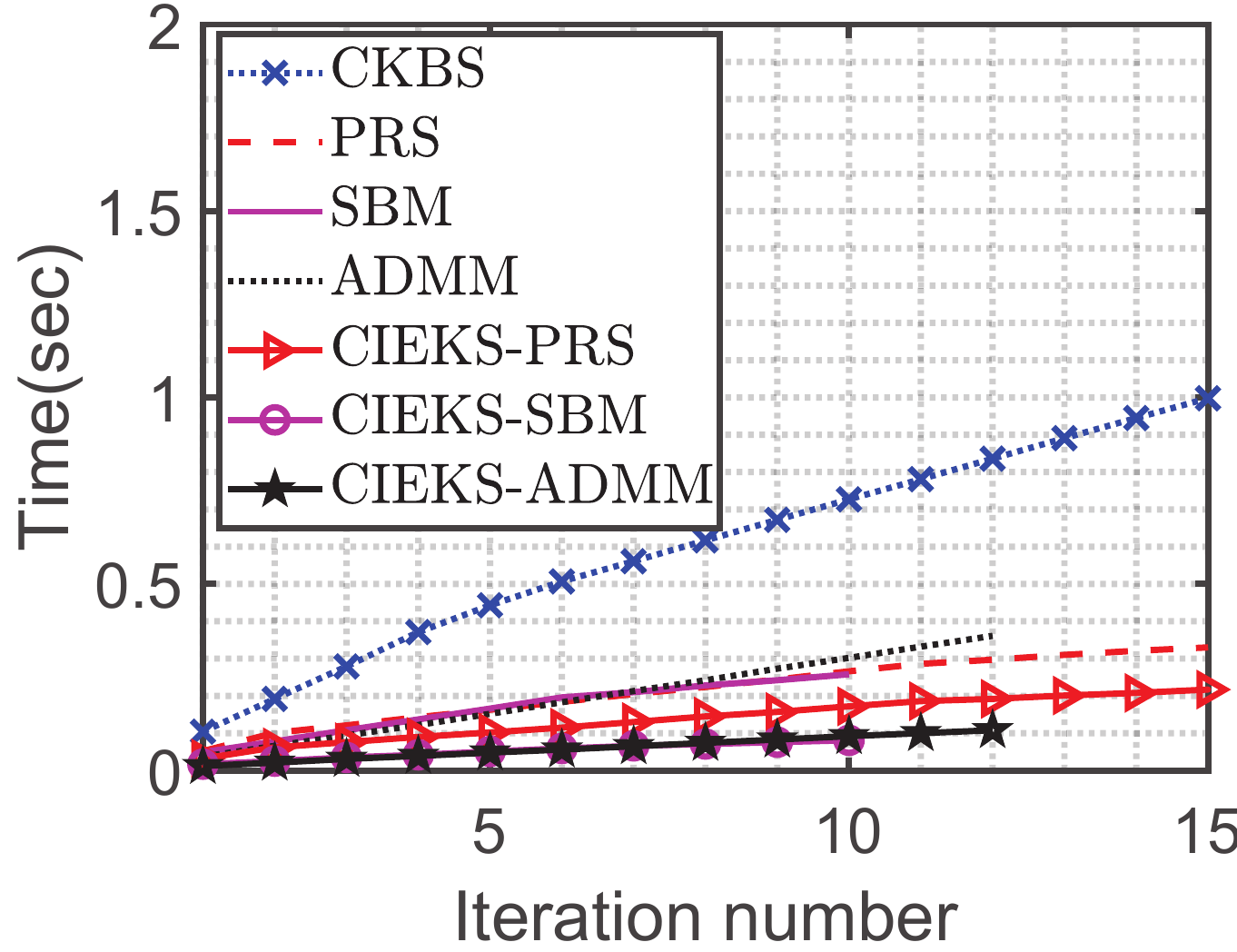}}
	\end{minipage}
	\caption{The cost function values and the running times (sec) of all the methods versus the iteration number. The $\mathbf{x}$-subproblems arising in batch PRS, SBM, and ADMM are solved with Gauss--Newton. }
	\label{fig:compared_result}
\end{figure}
	\setlength\tabcolsep{1.2pt}
	\begin{table}[!tbh]
	\centering
	\caption{Average running time (sec) at different time step counts $T$.}
	\begin{tabular}{|c|ccc|c|ccc|}
		\hline
	    { $T$} 	& PRS     & SBM     & ADMM    & CKBS & CIEKS-PRS   & CIEKS-SBM  & CIEKS-ADMM  \\ \hline
	    $10^3$  & 569.9   & 375.8   & 382.5   &23    & 9.4          & 3.4      & 4.0      \\
		$10^4$  & 3256    & 2193    & 2284    &186   & 89.3         & 26.7     & 31.1      \\
	    $10^5$  & --      & --      & --      &1506  & 847          & 212      & 298        \\
		$10^6$  & --      & --      & --      &--     & 8121           & 2057     & 2913  \\  \hline
	\end{tabular}\label{table1}
\end{table}

\section{Conclusion}
\label{sec:conclusion}
In this paper, we have developed a general framework for constructing constrained smoother-based variable splitting methods, which can be used to solve state-estimation problems with nonlinear equality and inequality constraints. The solution is computationally efficient, because of leveraging the Markov structure of the problem.
The experiments have been used to demonstrate the computational benefits.

\bibliographystyle{IEEEtran}
\bibliography{refs}

\end{document}